\documentclass[12pt, a4 paper]{amsart} 

\usepackage[psamsfonts]{amssymb} 
\usepackage{amsfonts,amsmath,mathabx,mathrsfs}
\usepackage{color}

\usepackage{amsthm, amssymb}
\usepackage{amsfonts} 
\usepackage{epsfig,multicol} 

\usepackage{xypic}
\usepackage{hyperref}
\usepackage{psfrag}

\input xy
\xyoption{all}

\makeatletter
\newtheorem*{rep@theorem}{\rep@title}
\newcommand{\newreptheorem}[2]{%
\newenvironment{rep#1}[1]{%
 \def\rep@title{#2 \ref{##1}}%
 \begin{rep@theorem}}%
 {\end{rep@theorem}}}
\makeatother

\newtheorem{quprime}{Question}

\newtheorem{theorem}{Theorem}[section] 

\newtheorem{proposition}[theorem]{Proposition}
\newtheorem{lemma}[theorem]{Lemma}

\newtheorem{question}{Question}
\newtheorem{conjecture}[theorem]{Conjecture}
\newreptheorem{theorem}{Theorem}
\newreptheorem{proposition}{Proposition}
\newreptheorem{lemma}{Lemma}

\theoremstyle{remark}

\theoremstyle{definition}
\newtheorem{definition}[theorem]{Definition}

\def\N{\mathbb N}

\def\G{\Gamma}

\def\S{\mathbb S}

\def\-{\overline}
\def\wh{\widehat}

\def\epi{{\rm{Epi}}}

\def\P{\mathcal{P}}
\def\VH{\mathcal{VH}}

\def\G{\Gamma}

\def\<{\langle}
\def\>{\rangle}

\begin{document}

\title[On recognizing virtually special cube complexes]{On the recognition problem for virtually special cube complexes}

\author[Martin R. Bridson]{Martin R.~Bridson}
\address{Martin R.~Bridson, Mathematical Institute, University of Oxford, Andrew Wiles Building, Oxford, OX2 6GG, UK}
\email{bridson@maths.ox.ac.uk} 
 
\author[Henry Wilton]{Henry Wilton}
\address{Henry Wilton, DPMMS, Centre for Mathematical Sciences, Wilberforce Road, Cambridge, CB3 0WB, UK}
\email{h.wilton@maths.cam.ac.uk}

\subjclass[2010]{20F10, 20F67, 57M07}

\thanks{Both authors are supported by the EPSRC. Bridson is also supported by a Wolfson Research Merit Award from the Royal Society.}

\maketitle

\def\G{\Gamma}
\def\epi{{\rm{Epi}}}

\begin{abstract}
We address the question of whether the property of being virtually special (in the sense of Haglund and Wise) is algorithmically decidable for finite, non-positively curved cube complexes.  Our main theorem shows that it cannot be decided  by examining one hyperplane at a time. Specifically, we prove that there does not exist an algorithm that,
given a compact non-positively curved squared 2-complex $X$ and a hyperplane $H$ in $X$, will decide whether or not there is a finite-sheeted cover of $X$ in which no lift of $H$ 
self-osculates. 
\end{abstract}

\section{Introduction}

Haglund and Wise's concept of a \emph{special cube complex} \cite{haglund_special_2008} has had profound ramifications in geometric group theory and low-dimensional topology, most notably in Agol's solution to the Virtually Haken Conjecture \cite{agol_virtual_2013}.  A compact, non-positively curved cube complex is {\em special} if it admits a local isometry
to the canonical cube complex associated to a right-angled Artin group. When phrased thus, this is not a condition that
one expects to be able to check in a practical manner, given a compact cube complex $X$. But Haglund and Wise proved that being special is
equivalent to a purely combinatorial condition which requires that each of the finitely many hyperplanes in $X$ should not behave in one of a small number of forbidden ways:
a hyperplane should not cross itself, it should be 2-sided, it should not self-osculate, and it should not inter-osculate with
another hyperplane (see \cite{haglund_special_2008} for details). We say that a
hyperplane is {\em clean} if it does not exhibit any of these forbidden behaviours.  

Given the combinatorial nature of these conditions, it is easy to check if a given hyperplane is clean
and hence whether a given compact cube complex is special or not. However, in applications, it is not specialness that matters so
much as {\em virtual specialness},  i.e. whether the complex
has a finite-sheeted covering space that is special.  
Agol \cite[p. 33]{agol_virtual_????} has promoted the study of the following question.

\begin{question}\label{qu: VS decidable}
Is there an algorithm that, given a  finite, non-positively curved cube complex, will determine
whether or not that complex is virtually special?
\end{question}

This is equivalent to the question:

\begin{quprime}\label{qu: VS decidable'}
Is there an algorithm that, given a  finite, non-positively curved cube complex $X$ (whose
hyperplanes are $H_1,\dots,H_n$, say), will determine the existence of a finite-sheeted covering
$\widehat{X}\to X$ in which each lift of each $H_i$ is clean.
\end{quprime}

One has to think clearly for a moment to see that this question
is not equivalent to:

\begin{question}\label{qu: finite collection}
Is there an algorithm that, given a  finite, non-positively curved cube complex $X$ and a collection
of hyperplanes $H_1,\dots,H_n$, will determine the existence of a finite-sheeted covering
$\widehat{X}\to X$ in which each lift of each $H_i$ is clean.
\end{question}

One's first thought is that a positive solution to Question 1${}^\prime$ should proceed
via a positive solution to Question \ref{qu: finite collection}. But on reflection one realises that the
latter would be significantly more powerful because certain of the complexes $X$ for which one
did establish the existence of $\widehat X$ would not be virtually special. Focusing on this difference,
we shall answer Question \ref{qu: finite collection} in the negative.  The following theorem is an immediate consequence of our main technical result, Theorem \ref{thm: Main theorem}.

\begin{theorem} \label{thm: Intro main}
There does not exist an algorithm that, given a finite, non-positively curved cube complex $X$ and a  
hyperplane $H$ in $X$, can determine whether or not  there exists a finite-sheeted covering
$\widehat{X}\to X$ in which each lift of $H$ is clean.
\end{theorem}

Note that this theorem is equivalent to the corresponding statement with `each' replaced by `some': this is clear if $\widehat{X}\to X$ is normal, and since the lift of a clean hyperplane is clean, we are free to pass to a finite-sheeted normal covering.

The main ingredient in our proof of Theorem \ref{thm: Main theorem} (and hence Theorem \ref{thm: Intro main}) is our work in \cite{bridson_triviality_2015}, in which we showed that there is no algorithm that determines whether or not a finite complex has a proper finite-sheeted cover (see Theorem \ref{thm: FP profinite triviality}).  In \cite{bridson_triviality_2015} we also proved that this undecidability persists in the setting of finite, non-positively curved cube complexes (see Theorem \ref{thm: NPC profinite triviality}). The point of this paper is to note that, with care, the same procedure can be made to translate the non-existence of a proper finite-sheeted cover into the failure of a certain hyperplane to be virtually clean.

The obstruction to cleanness that we focus on to prove Theorem \ref{thm: Main theorem} is self-osculation. Our construction shows that the question of whether one can remove self-osculations in finite-sheeted covers  is already undecidable in dimension 2, and this is the context in which we shall work for the remainder of this paper. Although this theorem does not settle Question \ref{qu: VS decidable}, it strongly constrains the nature of any possible positive solution: a putative recognition algorithm would have to take account of all the hyperplanes of $X$ simultaneously.  We therefore believe that Question \ref{qu: VS decidable} has a negative solution:

\begin{conjecture}
Virtual specialness is an undecidable property of non-positively curved cube complexes. 
\end{conjecture}

\section{$\mathcal{VH}$-complexes}

Rather than working with non-positively curved cube complexes in full generality, we restrict attention to the class of \emph{$\mathcal{VH}$-complexes}.

\begin{definition}
A square complex is \emph{$\mathcal{VH}$} if there is a partition of the edges of the 1-skeleton into \emph{vertical} and \emph{horizontal} subgraphs, such that the boundary cycle of each 2-cell crosses vertical and horizontal edges alternately.
\end{definition}  

We refer the reader to \cite{bridson_$scr_1999,wise_complete_2007} for the basic theory of $\mathcal{VH}$-complexes, but we recall some salient facts here.  In common with all cube complexes, much of the geometry of a $\mathcal{VH}$-complex $X$ is encoded in its set of \emph{hyperplanes} $Y\subseteq X$.  (A hyperplane is a maximal connected union of midcubes.)   We are concerned with certain regularity hypotheses on the hyperplanes.

\begin{definition}
Let $Y$ be a hyperplane embedded in a $\mathcal{VH}$-complex $X$. The union of all the open cells intersecting $Y$, together with the closest-point projection map, defines a bundle $\mathring{N}\to Y$ with fibre $(0,1)$, which has a natural completion to an $I$-bundle denoted by $N\to Y$ equipped with a map $\iota:N\to X$ extending the inclusion of $\mathring{N}$ into $X$. If $N$ is a trivial bundle then $Y$ is called \emph{non-singular} (or \emph{2-sided}).  In this case, $N\cong Y\times [0,1]$, and the two maps
\[
\partial_t:y\mapsto \iota(y,t)
\]
for $t=0,1$, are called the \emph{pushing maps} associated to $Y$.

The hyperplane $Y$ is called \emph{clean} if the pushing maps $\partial_0,\partial_1$ are both embeddings.
\end{definition}

If every hyperplane is non-singular then the complex $X$ is \emph{non-singular}, and if every hyperplane is clean then the whole complex $X$ is called \emph{clean}. The
following theorem of Haglund and Wise reconciles this terminology with that of virtually special complexes.

\begin{theorem}\cite[Theorem 5.7]{haglund_special_2008}
A $\mathcal{VH}$-complex is virtually special if and only if it is virtually clean.
\end{theorem}

Thus, we may specialize Question \ref{qu: VS decidable} to the case of $\mathcal{VH}$-complexes as follows.

\begin{question}\label{qu: VC decidable}
Is there an algorithm that decides whether or not a given finite $\mathcal{VH}$-complex is virtually clean? 
\end{question}

As in the introduction, our focus here is a \emph{local} version of this question.   We call a hyperplane $Y$ in a $\mathcal{VH}$-complex $X$ \emph{virtually clean} if there exists a finite-sheeted covering map $p:\widehat{X}\to X$ and a connected component $\widehat{Y}\subseteq p^{-1}Y$ such that the hyperplane $\widehat{Y}$ of $\widehat{X}$ is clean.  Our main theorem can now be stated as follows.

\begin{theorem}\label{thm: Main theorem}
There is a recursive sequence of pairs of compact $\mathcal{VH}$-complexes $X_n$ and hyperplanes $Y_n\subseteq X_n$ such that the set of natural numbers $n$ for which $Y_n$ is virtually clean is recursively enumerable but not recursive.
\end{theorem}

The complexes $X_n$ that we shall construct to prove Theorem \ref{thm: Main theorem} are never (globally) virtually clean---indeed, $\pi_1X_n$ is not residually finite---so they are not themselves candidates to answer Question \ref{qu: VC decidable}.

\section{Undecidability and finite covers}

The proof of Theorem \ref{thm: Main theorem} relies on our recent work on the triviality problem for profinite completions, in which we proved the following theorem \cite[Theorem B]{bridson_triviality_2015}.

\begin{theorem}\label{thm: NPC profinite triviality}
There is a recursive sequence of compact, non-positively curved square complexes $X_n$ such that the set
\[
\{n\in\mathbb{N}\mid \widehat{\pi_1X_n}\ncong 1\}
\]
is recursively enumerable but not recursive.
\end{theorem}

Here, $\widehat{\Gamma}$ denotes the profinite completion of a group $\Gamma$. 
By definition, $\widehat{\Gamma}$ is trivial if and only if every finite quotient of $\Gamma$ is trivial.

The proof of this theorem begins with a non-positively curved square complex $J$ that has no proper finite-sheeted covers.  With a careful choice of $J$, we can ensure that, in the output sequence $(X_n)$, each $X_n$ is a $\mathcal{VH}$-complex.   The main result of this section is Theorem \ref{thm: Tech}, a slight refinement of Theorem \ref{thm: NPC profinite triviality}. We will start by explaining the construction that proves Theorem \ref{thm: NPC profinite triviality}, and then observe some additional properties of the examples that we use to deduce Theorem \ref{thm: Tech}.

The first ingredient in the proof of Theorem \ref{thm: NPC profinite triviality} is the corresponding result for finitely presented groups (without any non-positive curvature condition).

\begin{theorem}\label{thm: FP profinite triviality}
There is a recursive sequence of finite presentations $\P_n$ of groups $\Gamma_n$ such that the set
\[
\{n\in\mathbb{N}\mid \widehat{\Gamma}_n\ncong 1\}
\]
is recursively enumerable but not recursive.
\end{theorem}

For $J$  we will use a complex that was constructed by Wise.

\begin{theorem}[Wise \cite{wise_complete_2007}]\label{thm: Wise CSC}
There exists a compact $\mathcal{VH}$-complex $J$ such that $\pi_1J$ is infinite but has no proper finite quotients.
\end{theorem} 

Burger and Mozes exhibited compact $\mathcal{VH}$-complexes with the even more striking property that their fundamental groups are infinite and simple \cite{burger_finitely_1997}.

We next perform a construction reminiscent of the well known Kan--Thurston construction.

\begin{definition}
Let $J$ be as above and fix a simple closed loop $c$ in a vertical component of the 1-skeleton of $J$.  Let $K$ be the presentation 2-complex associated to a finite presentation $\mathcal P \equiv \<a_1,\ldots,a_n\mid r_1,\ldots,r_m\>$ of a group $\Gamma$, and let $V:=K^{(1)}$.  We define $J_{\P}$ to be the space obtained by attaching $m$ copies of $J$ to $V$, with the $j$-th copy attached by a cylinder joining $c$ to the edge-loop in $V$ labelled $r_j$. More formally, writing $\rho_j: \mathbb{S}^1\to V$  for this last loop, we define $\sim$ to be the equivalence relation on 
\[
V\coprod \big(\mathbb{S}^1\times [0,1]\big)\times\{1,\dots,m\} \coprod J\times \{1,\dots,m\}
\]
defined by
\[
\forall t\in \mathbb{S}^1\, \forall j\in\{1,\dots,m\}\ :\  \rho_j(t)\sim (t,0,j) {\rm{~and~}} (t,1,j)\sim (c(t),j),
\]
and define $J_\P$ to be the quotient space.
\end{definition}

\begin{figure}[h]
\begin{center}
 \centering \def\svgwidth{300pt}
 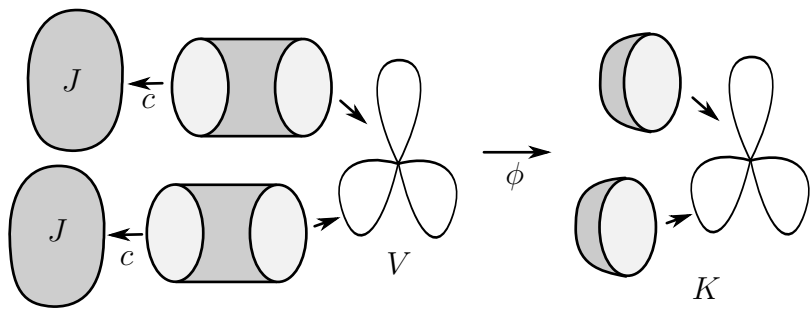
 \caption{The complex $J_\P$ and the map $\phi$.}
 \label{fig: J}
\end{center}
\end{figure}

The construction of $J_\P$ is illustrated in Figure \ref{fig: J}.  The next proposition collects together some easy properties of this construction.

\begin{proposition}\label{prop: Properties of J}
Let $J_\P$ be as above.  Then we have the following properties:
\begin{enumerate}
\item $J_\P$ has the structure of a $\mathcal{VH}$-complex;
\item there is a natural continuous map $\phi:J_\P\to K$, with the property that every homomorphism from $\pi_1J_\P$ to a finite group factors through $\phi_*$;
\item the composition
\[
\pi_1V\to \pi_1J_\P\stackrel{\phi_*}{\to}\Gamma
\]
is surjective.
\end{enumerate}
\end{proposition}
\begin{proof}
To prove (1), we subdivide the copies of $J$ and $V$ so that the attaching maps of the cylinders have the same length at either end. The gluing cylinders can then be subdivided into squares, giving the whole space the structure of a square complex, which is obviously $\mathcal{VH}$.

The map $\phi$ is defined to crush each copy of $J$ to a point; since a cylinder with one end identified with a point is homeomorphic to a disc, the resulting space is homeomorphic to $Y$.    The kernel of $\phi_*$ is normally generated by the copies of $\pi_1J$.  Since $\pi_1J$ has no finite quotients, if $f$ is a homomorphism from $\pi_1J_\P$ to a finite group, then $f$ kills the conjugacy classes defined by the loops in each copy of $J$, and therefore factors through $\phi_*$.  This proves (2).

Since $V$ is the entire 1-skeleton of $K$, item (3) is immediate.
\end{proof}

\begin{definition}
A \emph{pointed $\mathcal{VH}$-pair} is a triple $(L,V,*)$ where $L$ is a compact $\VH$-complex, $V$ is a vertical component of the 1-skeleton and $*$ is a vertex of $V$.
\end{definition}

\begin{theorem}\label{thm: Tech}
There is a recursive sequence of pointed $\mathcal{VH}$-pairs $(L_n,V_n,*_n)$ such that the following set  is recursively enumerable but not recursive:
\[
\{(n,\gamma)\mid n\in\N, \hat{\gamma}\neq 1\mathrm{~in~}\widehat{\pi_1L}_n\}
\]
where $\gamma$ runs through simple loops in $V_n$ based at $*_n$ and $\hat{\gamma}$ denotes the element of $\wh{\pi_1L}_n$ defined by $\gamma$.
\end{theorem}
\begin{proof}
Let $\P_n$ and $\Gamma_n$ be as in Theorem \ref{thm: FP profinite triviality} and let $K_n$ be the presentation complex for $\P_n$.  Set $L_n=J_{\P_n}$, let $V_n$ be the vertical subgraph of $J_{\P_n}$ from Proposition \ref{prop: Properties of J} and let $*_n$ be the unique vertex of $V_n$.  Note that the construction of $(L_n,V_n,*_n)$ from $\P_n$ is algorithmic, and therefore the sequence is recursive.

Consider now a finite set of simple based loops $\gamma_1,\ldots,\gamma_k$ that generate $\pi_1(V_n,*_n)$.  If there were an algorithm to determine whether or not such a loop $\gamma$ survives in some quotient of $\pi_1(L_n,*_n)$ then, applying the algorithm to each $\gamma_i$ in turn, we would obtain an algorithm to determine whether or not $\pi_1(V_n,*_n)$ has non-trivial image in some finite quotient of $\pi_1(L_n,*_n)$. But the latter occurs if and only if $\Gamma_n$ has a non-trivial finite quotient (by items (2) and (3) of Proposition \ref{prop: Properties of J}), which is undecidable.

This proves that the set of pairs $(n,\gamma)$ for which $\gamma$ survives in some finite quotient of $\pi_1(L_n,*_n)$ is not recursive. However, if $\gamma$ \emph{does} survive in some finite quotient then a systematic search will eventually discover this. Therefore, the set of such $(n,\gamma)$ is recursively enumerable.
\end{proof}

It will be convenient in what follows to have a fixed enumeration of the pairs $(n,\gamma)$ with $\gamma$ a simple loop in $V_n$ based at $*_n$: denote this by $m\mapsto (n_m,\gamma_m)$.  Correspondingly,  we write $(L_m,V_m,*_m)$ in place of $(L_{n_m},V_{n_m},*_{n_m})$, so $\gamma_m$ is a simple loop in $V_m$ based at $*_m$.

\section{Proof of Theorem \ref{thm: Main theorem}}

In this final section, we modify the output of Theorem \ref{thm: Tech} in order to deduce Theorem \ref{thm: Main theorem}.  First we construct $L'_n$ from $L_n$ by attaching a vertical loop $\alpha$ (consisting of a single edge) at $*_n$.  We also let $V'_n=V_n\cup\alpha$.  For each simple based loop $\gamma$ in $V_n$ we define $\gamma'$ to be the concatenation $\gamma'\cdot\alpha$ in $L'_n$.  
  
Each loop $\gamma'_n$ is a local isometry
\[
\gamma'_n:\S^1\to L'_n
\]
which induces a cellular structure on $\S^1$, sending vertices to vertices and edges to edges.   To construct the complex $X_n$, take two copies of $L'_n$, and attach an annulus $A_n\cong \S^1\times I$, where the two ends of the cylinder $A_n$ are attached by the two copies of the map $\gamma'_n$.  The construction is illustrated in Figure \ref{fig: X}.

\begin{figure}[h]
\begin{center}
 \centering \def\svgwidth{300pt}
 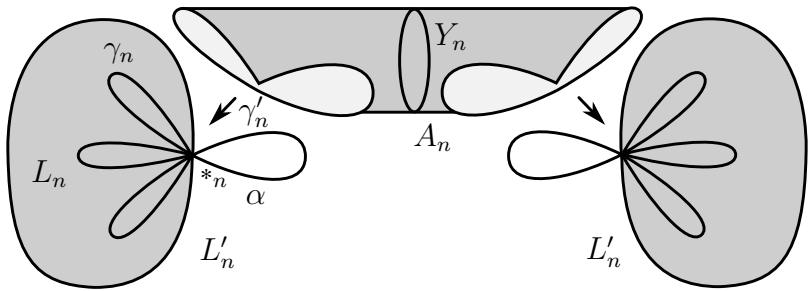
 \caption{The construction of the complex $X_n$.}
 \label{fig: X}
\end{center}
\end{figure}

Since the loops $\gamma'$ are vertical and locally geodesic, we can extend the $\mathcal{VH}$-structures of the two copies of $L'_n$  to a non-singular $\mathcal{VH}$-structure on $X_n$ by dividing the annulus $A_n$ into squares.

The hyperplane $Y_n$ is defined to be the unique vertical hyperplane in the annulus $A_n$. It is homeomorphic to a circle, and two-sided (i.e.\ non-singular) by definition. 

Our final lemma relates the geometry of $Y_n$ in finite-sheeted covers of $X_n$ to the group-theoretical properties of $\phi_*(\gamma_n)\in\Gamma_n$ (in the notation of Proposition \ref{prop: Properties of J} with $L_n=J_{\P_n}$).

\begin{lemma}\label{lem: Equivalence}
The following statements are equivalent.
\begin{enumerate}
\item The hyperplane $Y_n\subseteq X_n$ is virtually clean.
\item There is a based, finite-sheeted covering space $(R,\dagger)\to (L_n,*_n)$ so that the lift at $\dagger$  of $\gamma_n$ is not a loop.
\item There is a subgroup $N$ of finite index in $\pi_1L_n$ such that $\gamma_n\notin N$. 
\item The group element $[\gamma_n]$ survives in some finite quotient of $\pi_1L_n$.
\item The element $\phi_*([\gamma_n])$ survives in some finite quotient of $\Gamma_n$. 
\end{enumerate}
\end{lemma}
\begin{proof}
We first show that (1) implies (2).  Suppose that $p:Z\to X_n$ is a finite-sheeted covering map and that some component $W$ of $p^{-1}Y_n$ is clean. Consider the pushing map $\partial_0:W\to Z$. Choose some $\dagger\in p^{-1}(*_n)$ that lies in $\partial_0(W)$.  Let $R$ be the component of $p^{-1}L_n$ that contains $\dagger$.  Because $\partial_0$ is an embedding, it follows that the lift of $\gamma_n$ at $\dagger$ is not a loop.

We next show that (2) implies (1); to this end, we define a certain retraction $\rho:X_n\to L_n$.  Consider the natural retraction $X_n\to L'_n$ which identifies the two copies of $L'_n$ and maps the gluing annulus via $\gamma'_n$.  Next, consider the retraction $L'_n\to L_n$ which maps $\alpha(t)$ to $\overline{\gamma}_n(t)$, where $\overline{\gamma}_n$ is the path $\gamma_n$ in the opposite direction. The concatenation
\[
X_n\to L'_n\to L_n
\]
defines the retraction $\rho$.  To show that (2) implies (1), we pull the covering map $R\to L_n$ back along $\rho$ to obtain a covering space $Z\to X_n$.  Let $W$ be the unique component of the preimage of $Y_n$ such that $\dagger$ is contained in $\partial_t(W)$ for some $t=0$ or $1$.  By construction, $W$ is a circle and $\partial_t$ is given by the lift of $\gamma'_n=\gamma_n\cdot\alpha$.  Since the lift of $\gamma_n$ is an embedded arc with distinct endpoints, the lift of $\gamma'_n$ is an embedded loop, as required.

Items (2) and (3) are equivalent by standard covering-space theory: in one direction, take $N=\pi_1(R,\dagger)$, and in the other take $(R,\dagger)$ to be the based covering space corresponding to the subgroup $N$.

The equivalence of items (3) and (4) is also standard.  To show that (4) implies (3), just take $N$ to be the kernel of the finite quotient map.  For the converse, the action of $\pi_1L_n$ by left translation on $\pi L_n/N$ defines a homomorphism $\pi_1L_n\to \mathrm{Sym}(\pi L_n/N)$ in which $\gamma_n$ survives.

Finally, that (4) and (5) are equivalent follows from item (2) of Proposition \ref{prop: Properties of J}.
\end{proof}

Lemma \ref{lem: Equivalence} tells us that deciding if the hyperplane $Y_n\subseteq X_n$ is virtually clean is equivalent to deciding if $(n,\gamma_n)$ lies in the set considered in Theorem \ref{thm: Tech}; and this set is recursively enumerable but not recursive. This completes the proof of Theorem \ref{thm: Main theorem}.

\bibliographystyle{plain}

\end{document}